\documentclass[11pt]{article}
\usepackage{amsmath, amsthm, amssymb, color}

\usepackage[T2A]{fontenc}
\usepackage[cp1251]{inputenc}
\usepackage[english]{babel}
\usepackage{url}
\usepackage[normalem]{ulem} 

\newtheorem{defn}{Definition}[section]
\newtheorem{lemma}[defn]{Lemma}
\newtheorem{thm}[defn]{Theorem}
\newtheorem{prop}[defn]{Proposition}

\theoremstyle{remark}
\newtheorem{rem}[defn]{Remark}

\newtheorem{ex}[defn]{Example}

\newcommand{\h}{{\cal H}}
\newcommand{\mn}{\mathbb N}

\def\h{{\cal H}}

\def\bp{\begin{proof}}
\def\ep{\end{proof}}
\def\<{\langle}
\def\>{\rangle}
\newcommand{\la}{\langle}
\newcommand{\ra}{\rangle}
\newcommand{\seqgr[3]}{(#1_n)_{n=#2}^#3}
\newcommand{\seq[1]}{(#1_n)_{n=1}^\infty}
\newcommand{\seqEc[1]}{(#1_n)_{n\in E^c}}
\newcommand{\seqE[1]}{(#1_n)_{n\in E}}
\def\newin {\,\kern-0.4em\in\kern-0.15em}

\usepackage{makecell}

\title{Dual frames compensating for erasures - \\ non-canonical case}
\author{Ljiljana Aramba\v si\' c$^{\rm \, a}$ and Diana Stoeva$^{\rm \, b}$\\ \\
$^{\rm \, a}$ {\small Department of Mathematics, Faculty of Science, University of Zagreb,} \\
{\small Bijeni\v cka cesta 30, 10000 Zagreb, Croatia}\\
{\small Email: arambas@math.hr}\\
$^{\rm \, b}$ {\small Institute of Telecommunications, Technische Universit\"{a}t Wien ,}\\{\small Gusshausstrasse 25/E 389, 1040 Vienna, Austria}\\
{\small Email: diana.stoeva@tuwien.ac.at}
}

\begin{document}
\maketitle \pagestyle{myheadings}

%%%%%%%%%%%%%%%%%%%%%%%%%%%%%%%%5

\begin{abstract}
In this paper we study the problem of recovering a signal from frame coefficients with erasures. Suppose that erased coefficients are indexed by a finite set $E$.
Starting from a frame $\seq[x]$ and its arbitrary dual frame, we give sufficient conditions for constructing a dual frame of $(x_n)_{n\in E^c}$ so that the perfect reconstruction can be obtained from the preserved frame coefficients.
The work is motivated by methods using the canonical dual frame of $\seq[x]$, which however do not extend automatically to the case when the canonical dual is replaced with another dual frame.
The differences between the cases when the starting dual frame is the canonical dual and when it is not the canonical dual are investigated.
We also give several ways of computing a dual of the reduced frame, among which we are the most interested in the iterative procedure for computing this dual frame.
Computational tests show that in certain cases the iterative algorithm performs faster than the other considered procedures.
\end{abstract}



%%%%%%%%%%%%%%%%%%%%%%%%%%%%%%

\section{Introduction and notation} \label{intr}

Throughout the paper, $\h$ usually denotes a separable infinite-dimensional Hilbert space. However, with appropriate adjustments, all results  also apply to
finite-dimensional Hilbert spaces and justify the examination of computational efficiency of the provided algorithms in the finite-dimensional case. 
By $\Bbb B(\h)$ we denote the algebra of all bounded linear operators on $\h.$

A sequence $(x_n)_{n=1}^{\infty}$ in
$\h$ is a \emph{frame for $\h$} \cite{DS} if there exist positive constants $A$ and $B$ such that
\begin{equation}\label{frame def}
A\|h\|^2\leq \sum_{n=1}^{\infty}|\la h,x_n\ra |^2\leq B\|h\|^2,\quad h \in \h.
\end{equation}
If $A$ and $B$ can be chosen to be $1$,  $(x_n)_{n=1}^{\infty}$ is called a \emph{Parseval frame for $\h$}.
A sequence $(x_n)_{n=1}^{\infty}$ in $\h$ is a \emph{Bessel sequence in $\h$} if it satisfies the right hand side inequality in \eqref{frame def}.
For a Bessel sequence $(x_n)_{n=1}^{\infty}$ in $\h$, one defines the \emph{analysis operator} $U: \h\rightarrow \ell^2$ by $Uh=(\langle h,x_n\rangle)_{n=1}^\infty,\,h\in \h$. The adjoint operator $U^*$ is given by $U^*((c_n)_{n=1}^\infty)=\sum_{n=1}^{\infty}c_nx_n$ for $(c_n)_{n=1}^\infty\in \ell^2$.

Let $(x_n)_{n=1}^\infty$ be a frame for $\h$.
Then there exists a frame $\seq[z]$ for $\h$ so that the reconstruction formula
\begin{equation*}\label{defaltdual}
h=\sum_{n=1}^{\infty}\la h,x_n\ra z_n,\quad h \in \h,
\end{equation*}
holds; such $\seq[z]$ is called a \emph{dual frame of $\seq[x]$}.
The \emph{frame operator} $U^*U\in \Bbb B(\h)$ is invertible (i.e., bounded and bijective) and the sequence $(y_n)_{n=1}^\infty=((U^*U)^{-1}x_n)_{n=1}^{\infty}$ is a dual frame of $\seq[x]$, called \emph{the canonical dual frame} (in short, \emph{the canonical dual})
\emph{of $\seq[x]$}.
When the frame $(x_n)_{n=1}^\infty$ is not a Schauder basis for $\h$ (called an \emph{overcomplete} or \emph{redundant frame}), there are other dual frames in addition to the canonical dual. For more on general frame theory we refer e.g. to \cite{CKu, Cbook, HL, Hbook, KC}.

The reconstruction property of frames and the possibility for redundancy are some of the main reasons which make frames so important and with wide applications (e.g., in signal processing, data compression, optics, signal detection, and many other areas).
The redundancy makes possible a perfect reconstruction from frame coefficients with erasures, that is, when some coefficients are lost or damaged (which is often the case e.g. in signal transmission), assuming that preserved coefficients $\la h, x_n\ra$ arise from frame elements $x_n$ which span the space $\h.$
Related to this, we consider the following definition, introduced in \cite{LS}:

Let $(x_n)_{n=1}^\infty$ be a frame for $\h$.
It is said that a finite set of indices $E=\{i_1, i_2, \ldots,i_k\}\subset \mn$ satisfies the \emph{minimal redundancy condition} (in short, \emph{MRC}) \emph{for} $(x_n)_{n=1}^\infty$ if the linear span of the set $\{x_n:n\in E^c\}$ is dense in $\h,$ that is, if
\begin{equation}\label{def_mrp}
\overline{\text{span}}\,\{x_n:n\in E^c\}=\h.
\end{equation}
 As observed in \cite{LS}, based on \cite[Theorem~5.4.7]{Cbook},   \eqref{def_mrp} holds if and only if $(x_n)_{n\in E^c}$ is a frame for $\h$. 
This means that, for such a set $E$, if we take a dual frame $(v_n)_{n\in E^c}$ of $(x_n)_{n\in E^c}$, then for each $h\in \h$ it holds  $h=\sum_{n\in E^c}\la h,x_n\ra v_n $, that is, $h$ can be reconstructed even if we do not know the coefficients $(\la h,x_n\ra)_{n\in E}$.
This also means that, if $E$ does not satisfy the MRC for $\seq[x],$ then there is a nonzero vector $h\in \h$ orthogonal to $\{x_n:n\in E^c\},$
so $h$ cannot be reconstructed by using only the coefficients $(\la h,x_n\ra)_{n\in E^c}$. This explains the word \emph{minimal} in the
definition of the MRC.

Let $(x_n)_{n=1}^\infty$ be a frame for $\h$ and let the set of frame coefficients
\sloppy $(\la h,x_n\ra)_{n\in  E}$ of $h\in\h$ be lost, where $E$ is a finite set satisfying the \emph{MRC} for $(x_n)_{n=1}^\infty$. There are several approaches in the literature aiming recovery of $h$. One of them focuses on recovery of the lost coefficients $(\la h,x_n\ra)_{n\in  E}$, see, e.g., the bridging method in \cite{LS} or \cite{HanSun}.
Another approach deals with methods for inversion of the partial reconstruction operator  $R_E\in \Bbb B(\h)$ determined by  $R_Eh=\sum_{n\in E^c}\<h,x_n\>z_n$ (in cases when it is invertible), where $(z_n)_{n=1}^\infty$ means a dual frame of $(x_n)_{n=1}^\infty$, see \cite{LS}.
A third approach focuses on constructions of a dual frame of the reduced frame $(x_n)_{n\in E^c}$. This approach is considered in \cite{AB}, where the authors construct the canonical dual $(x_n)_{n\in E^c}$ based on the canonical dual of $(x_n)_{n=1}^\infty$.
Of course, a natural way to determine the canonical dual of $(x_n)_{n\in E^c}$  would be to use the definition of the canonical dual, which however might not be very efficient computationally in high dimensional spaces as it involves inversion of the respective frame operator.  This has been a motivation behind searching for methods that reduce the dimension, in particular focusing on the erasure set $E$, e.g. \cite{AB, LS}.

Instead of finding ways how to perfectly reconstruct a vector when some coefficients are lost in the process of transmission, some authors deal with the problem of finding optimal dual frames for erasures, that is, those dual frames which minimize the error of the reconstruction from the preserved coefficients (see \cite{HP, LH, LoH}). We also refer to \cite{CK} where the authors classify frames that are robust to a fixed number of erasures,
as well as to  \cite{ACM, ABfull} where full spark frames were discussed (a frame  $(x_n)_{n=1}^N$ for an $r$-dimensional Hilbert space is full spark if every set of indices $E$ with cardinality at most $N-r$ satisfies the MRC for $(x_n)_{n=1}^N$).

In this paper we focus on the aforementioned third approach for recovery, namely, on constructions of a dual frame of the reduced frame $(x_n)_{n\in E^c}$.
One of our main purposes is the construction of a dual frame $(v_n)_{n\in E^c}$ of the reduced frame $(x_n)_{n\in E^c}$ (not necessarily the canonical dual) based on an arbitrary dual frame of $(x_n)_{n=1}^\infty$. This is a sequel of the research from \cite{AB} where the same problem was investigated with a restriction that a starting dual frame of $(x_n)_{n=1}^\infty$ is the canonical dual, and where the constructed dual frame of the reduced frame is the canonical one.
In general, the canonical dual of a given frame
has some nice properties
(for example, it minimizes the coefficients in frame expansions via the given frame \cite[Lemma VIII]{DS}).
However, it might not be very appropriate to be used in applications as it might be computationally inefficient and, for example, in the case of Gabor frames it may fail some other nice desired properties like compactness, smoothness, and good time-frequency localization.
For this reason, explicit constructions and characterizations
 of other dual frames with some desired properties (in particular, without involving an operator inversion) have been of big interest in the last two decades in frame theory (see, e.g., \cite{BGS, CGabor, CKdualGabor, CKK, GJ, sc}).
This was also a motivation behind our work on construction of non-canonical dual frames for the erasure problem and on comparison with the canonical case.

\vspace{.1in}

The paper is organized as follows.

In Section~\ref{section2} we discuss the canonical case, that is, we start with a frame $\seq[x]$, its canonical dual
 $\seq[y]$, and a finite set of indices $E$ satisfying the MRC for $\seq[x]$ (and automatically for $\seq[y]$). We summarize results obtained in \cite{AB} about different ways for computing the canonical dual
  of the reduced frame $(x_n)_{n\in E^c}$, among which we are the most interested in the iterative procedure for computing this dual frame.
In the case of a finite frame $(x_n)_{n=1}^N$, we implemented the considered algorithms and tested them for efficiency in time-computation. More precisely, we compared the time for computation of the canonical dual of $(x_n)_{n\in E^c}$ via the code  based on the iterative algorithm (Proposition \ref{iterations_canonical}), via the code based on the procedure involving matrix inversion (Theorem \ref{thm2.5}(ii)), and via the pseudo-inverse approach using the MATLAB pinv-function. The tests noted in Section 4 show that in certain cases the iterative procedure is the fastest one among the considered approaches.

In Section~\ref{section3} we investigate the non-canonical case - we start with an arbitrary dual frame $\seq[z]$ for $\seq[x].$ The  question we deal with is: can we, by imitating the canonical case, obtain analogous ways for computing a dual frame of the reduced frame? In this case the assumption that a finite set of indices $E$ satisfies the MRC for $\seq[x]$ does not guarantee that adaptations of formulas from the canonical case are well defined; this is not surprising, since the relation between a frame and its canonical dual is much stronger than the relation between a frame and its arbitrary dual. However, with some additional assumptions, the adapted formulas are well defined and they give us a dual frame for $(x_n)_{n\in E^c}$, not necessarily the canonical one.
Again, we implemented and tested the considered algorithms 
for efficiency in time-computation. In particular, we also compared the performance of the algorithms
in the canonical and in the non-canonical case
for the same frame and the tests show that in certain cases the use of a non-canonical dual can be faster than the use of the canonical one. Results from the tests are noted in Section 4 and they justify our interest to the non-canonical case.

For convenience of the writing and without loss of generality, through the entire paper we will write $E=\{1,2,\ldots,k\}$ instead of $E=\{i_1, i_2, \ldots,i_k\}$.
The identity operator on $\h$ will be denoted by $I_{\h}$ or simply by $I$ if there is no risk of confusion. For $x,y\in \h$, $\theta_{y,x}$ will denote
 the rank one operator defined by $\theta_{y,x}(h)=\la h,x\ra y,$ which is  clearly in $\Bbb B(\h)$.

%%%%%%%%%%%%%%%%%%%%%%%%%%%%%%%%%%%%%%%%%%%%%%%%%%%%%%%%%%%%%%%%%%%%%%%%%%%%%%%%%%%%%%%%%%%%%%%%%
\section{Construction by using the canonical dual}\label{section2}

Let $\seq[x]$ be a frame for $\h$ and let $E$ be a set of indices that satisfies the MRC for $(x_n)_{n=1}^\infty.$ In \cite{AB} the authors use the canonical dual
 $\seq[y]$ of $\seq[x]$ to construct the canonical dual
  $(v_n)_{n\in E^c}$ of $(x_n)_{n\in E^c}$.
In the following theorem we summarize results obtained in
\cite[Proposition~2.4, Theorems~2.5 and 2.12, and (2.29)]{AB}, 
where two ways of presenting the vectors $v_n,n\in E^c,$ are given.

\begin{thm} \label{thm2.5}
Let $(x_n)_{n=1}^\infty$ be a frame for $\h$ and let $(y_n)_{n=1}^\infty$ denote the canonical dual of $(x_n)_{n=1}^\infty$.
Suppose that a finite set of indices $E=\{1,2,\ldots,k\}$ satisfies the MRC for $(x_n)_{n=1}^\infty$. Then the following holds:

 \item[{\rm (i)}] The canonical dual 
  $(v_n)_{n\in E^C}$ of $(x_n)_{n\in E^C}$  can also be written as
\begin{equation}\label{candual-operator}
v_n=(I-\sum_{i=1}^k\theta_{y_{i},x_{i}})^{-1}y_n,\quad n\in E^c.
\end{equation}

\item[{\rm (ii)}]
For each $n \in E^c$,   the numbers $\alpha_{n1}, \alpha_{n2},\ldots, \alpha_{nk},$
given  by the formula
\begin{equation}\label{XVIII}
{\small
\left[
\begin{array}{c}
\alpha_{n1}\\
\alpha_{n2}\\
\vdots\\
\alpha_{nk}
\end{array}
\right]=
\left(\left[\begin{array}{cccc}
\langle y_{1},x_{1}\rangle&\langle y_{2},x_{1}\rangle&\ldots&\langle y_{k},x_{1}\rangle\\
\langle y_{1},x_{2}\rangle&\langle y_{2},x_{2}\rangle&\ldots&\langle y_{k},x_{2}\rangle\\
\vdots&\vdots& &\vdots\\
\langle y_{1},x_{k}\rangle&\langle y_{2},x_{k}\rangle&\ldots&\langle y_{k},x_{k}\rangle
\end{array}
\right]-I\right)^{-1}
\left[
\begin{array}{c}
\langle y_n,x_{1}\rangle\\
\langle y_n,x_{2}\rangle\\
\vdots\\
\langle y_n,x_{k}\rangle
\end{array}
\right]
}
\end{equation}
are well defined and
the sequence $(v_n)_{n\in E^C}$ determined by
\begin{equation}\label{candual-matrix}
v_n=y_n-\sum_{i=1}^k\alpha_{ni}y_{i},\quad n\in E^c,
\end{equation}
  is the canonical dual
  of $(x_n)_{n\in E^C}$.
\end{thm}

In $r$-dimensional space $\h$, the construction of $(v_n)_{n\in E^c}$ via the inverse of the respective frame operator or via (\ref{candual-operator}) involves inversion of an $r\times r$ matrix, which 
would not be very efficient for computational purposes when $r$ is big.
This is a motivation behind the search for more efficient constructions, e.g. the construction via (\ref{XVIII})-(\ref{candual-matrix}) that involves inversion of just $k\times k$ matrix and thus expected to be more efficient in case $k$ is much smaller than $r$.
In \cite[Theorem~2.14]{AB} the authors also present an iterative procedure for computing the inverse of the operator $I-\sum_{i=1}^k\theta_{y_{i},x_{i}}$ from (\ref{candual-operator}) by computing inverses of simple operators of the form $I-\theta_{y,x}$. The procedure uses the well-known fact that $I-\theta_{y,x}$ is invertible on $\h$  if and only if $\la y,x\ra\neq 1$, and in the case of invertibility, the inverse is given by
\begin{equation}\label{xy-inverse}
(I-\theta_{y,x})^{-1}=I+\frac{1}{1-\<y,x\>}\theta_{y,x}.
\end{equation}
This lead us to an iterative procedure for
determining the canonical dual of $(x_n)_{n\in E^c}.$ Moreover, if $E=\{1,2,\ldots,k\},$ then the $j$-th iteration, for each $j=1,\ldots,k,$ gives the canonical dual
of $(x_n)_{n={j+1}}^\infty.$
In the following proposition we state this result, but we also include a direct proof of the statement, without using results from \cite{AB}.

\begin{prop}\label{iterations_canonical}
Let $(x_n)_{n=1}^\infty$ be a frame for $\h$ and let $(y_n)_{n=1}^\infty$ denote the canonical dual
 of $(x_n)_{n=1}^\infty$.
Suppose that the set $E=\{1,2,\ldots,k\}$ satisfies the MRC for $\seq[x]$. Denote
$$v_n^0:=y_n, \quad n\in\Bbb N.$$
For $j$ from $1$ to $k$ let
\begin{eqnarray}
\label{candual-iterations-alpha} \alpha_n^j:&=&\frac{\la v_n^{j-1},x_j\ra}{1-\la v_j^{j-1},x_j\ra}, \quad n\geq j+1,\\
\label{candual-iterations}v_n^j:&=&v_n^{j-1}+ \alpha_n^j v_j^{j-1}, \quad n\ge j+1.
\end{eqnarray}
Then, for every $j$ from $1$ to $k$, the sequences in \eqref{candual-iterations-alpha} and \eqref{candual-iterations} are well defined, and the sequence $(v_n^j)_{n=j+1}^\infty$ is the canonical dual 
 of $(x_n)_{n=j+1}^\infty.$
\end{prop}
\begin{proof}
Step $j=1.$ Let us show that the sequence $(v_n^1)_{n=2}^\infty$ is the canonical dual 
 of $(x_n)_{n=2}^\infty.$

In order to compute $\alpha_n^1$ for $n\ge 2$, we first have to verify that $\la y_1,x_1\ra\neq 1.$ 
If we assume that $\la y_1,x_1\ra =1$, then
\cite[Lemma IX]{DS} 
would imply that $(x_n)_{n=2}^\infty$ is not complete in $\h$ and so the set $\{1\}$ would not satisfy the MRC for $(x_n)_{n=1}^\infty$.
This would contradict the assumption of the proposition that $E$ satisfies the MRC for $(x_n)_{n=1}^\infty$, so it has to be that $\la y_1,x_1\ra\neq 1$ and thus we can proceed dealing with $\alpha_n^1$ and $v_n^1$.
For $n\ge 2$ we have
$$v_n^1=y_n+\frac{\la y_n,x_1\ra}{1-\la y_1,x_1\ra}y_1=\left(I+\frac{\theta_{y_1,x_1}}{1-\la y_1,x_1\ra}\right)y_n=\left(I-\theta_{y_1,x_1}\right)^{-1}y_n.$$
Let $S$ be the frame operator for $\seq[x]$. Then we have
\begin{eqnarray*}
\sum_{n\ge 2}\la h,x_n\ra x_n&=&\sum_{n=1}^\infty \la h,x_n\ra x_n-\theta_{x_1,x_1}(h)=(S-\theta_{x_1,x_1})h\\
&=&S(I-\theta_{S^{-1}x_1,x_1})h=S(I-\theta_{y_1,x_1})h
\end{eqnarray*}
for all $h\in\h.$ Therefore, $S(I-\theta_{y_1,x_1})$ is the frame operator for $(x_n)_{n=2}^\infty,$ so the canonical dual 
 of $(x_n)_{n=2}^\infty$ is
$$(S(I-\theta_{y_1,x_1}))^{-1}x_n=(I-\theta_{y_1,x_1})^{-1}S^{-1}x_n=(I-\theta_{y_1,x_1})^{-1}y_n=v_n^1, \quad n\ge 2.$$

Step $j= 2$.
Observe that the way how we obtain  $(v_n^2)$ from $(v_n^1)$ is the same as
the way we obtained $(v_n^1)$ from $(v_n^0).$
Namely, instead of $(x_n)_{n=1}^\infty,$  $(v_n^0)_{n=1}^\infty$, and the set $\{1,2,\ldots,k\}$, we now have $(x_n)_{n=2}^\infty,$  $(v_n^1)_{n=2}^\infty$, and the set $\{2,\ldots,k\}$ (since $(v_n^1)_{n=2}^\infty$ is the canonical dual of  $(x_n)_{n=2}^\infty$  and  $\{2,\ldots,k\}$ satisfies the MRC for $(x_n)_{n=2}^\infty$).
So, we can proceed in the same way as above to obtain the desired conclusions for $j=2$ and then further for $j=3, 4, \ldots, k$.
\end{proof}

\begin{rem}\label{d_implies_a}
The converse of the previous proposition holds in the sense that, if the sequences in \eqref{candual-iterations-alpha} and \eqref{candual-iterations} are well defined for some $j$, then the set $\{1,\ldots,j\}$ satisfies the MRC for $\seq[x].$
Indeed, $(\alpha_n^1)_{n\ge 2}$ (and consequently $(v_n^1)_{n\ge 2}$) is well defined precisely when $\la y_1,x_1\ra\neq 1,$ that is, precisely when $I-\theta_{y_1,x_1}$ is invertible. By \cite[Proposition~2.4]{AB}, $I-\theta_{y_1,x_1}$ is invertible if and only if $\{1\}$ satisfies the MRC for $\seq[x].$
In the same manner we see that $(\alpha_n^2)_{n\ge 3}$ is well defined if and only if $I-\theta_{v_2^1,x_2}$ is invertible, that is, if and only if $\{2\}$ satisfies the MRC for $(x_n)_{n=2}^\infty,$ which obviously happens precisely when $\{1,2\}$ satisfies the MRC for $\seq[x].$
We proceed inductively. \hfill $\Box$
\end{rem}

%%%%%%%%%%%%%%%%%%%%%%%%%%%%%%%%%%%%%%%%%%%%%%%%%%%%%%%%%%%%%%%%%%%%%%%%%%%%%%%%%%%%%%%%%%%%%%%%%%%
\section{Construction by using an arbitrary dual frame}\label{section3}

Let $\seq[x]$ be a frame for $\h,$ and let  $E$ be a finite set satisfying the MRC for $\seq[x]$.
In  Section~\ref{section2} we have presented three ways to find the canonical dual 
 $(v_n)_{n\in E^c}$ of $(x_n)_{n\in E^c}$: by inverting a matrix as in \eqref{XVIII}-\eqref{candual-matrix}, by inverting an operator as in \eqref{candual-operator}, and by iterations as in \eqref{candual-iterations}.
Note that for all these three methods,  one starts with the canonical dual 
 $\seq[y]$ of $\seq[x].$
(Of course, one can find $(v_n)_{n\in E^c}$ directly from $(x_n)_{n\in E^c}$ by computing the inverse of the frame operator of $(x_n)_{n\in E^c}$.
However, in 
 an $r$-dimensional space $\h$, this would lead to computing the inverse of a matrix of order $r$, and usually $r$
  is significantly larger than $k$ (the cardinality of $E$), which is the order of the mentioned matrix from \eqref{XVIII}.) 
It is natural to
consider the above mentioned methods but starting with an arbitrary dual frame $\seq[z]$ of $\seq[x]$ instead of the canonical one, inspired by \cite[Proposition 2.8 and Remark 2.9]{AB} and \cite{sc}.
Let us first give an example to motivate our study, that is, to show some differences which occur when turning from the canonical dual to another dual.

\begin{ex}\label{ex1}
Let $\seq[e]$ be an orthonormal basis of $\h.$ Consider the frame
$$\seq[x]=(e_1, e_1, e_1, e_2, e_3, e_4, \ldots)$$
and its (non-canonical) dual frame
$$\seq[z]=(\frac{1}{2}e_1, 0, \frac{1}{2} e_1, e_2, e_3, e_4,\ldots).$$

(a) First observe that the set $E=\{1,3\}$ satisfies the MRC for $\seq[x],$ but not for its dual $\seq[z].$
This cannot happen in the canonical case, since the canonical dual 
 is the image of the initial frame by an invertible operator and thus a finite set $E$ has the MRC for a given frame if and only if $E$ has the MRC for its canonical dual.

(b) The set $\{1\}$ satisfies the MRC for $\seq[x]$ and $\seq[z]$. We can apply formulas as in  \eqref{XVIII} and \eqref{candual-matrix}, replacing the canonical dual $\seq[y]$ with $\seq[z]$. This will give the sequence $(v_n)_{n=2}^\infty=(0,e_1, e_2, e_3, \ldots),$ which is a non-canonical dual 
 of $(x_n)_{n=2}^\infty$.

Doing the same for the set $\{1,2\}$, which also satisfies the MRC for $\seq[x]$ and $\seq[z]$, we obtain the sequence
$(v_n)_{n=3}^\infty=(e_1, e_2, e_3, \ldots),$ which is the canonical dual of $(x_n)_{n=3}^\infty$. This is not a surprise since $(x_n)_{n=3}^\infty$ is a Riesz basis for $\h$, so it has only one dual frame.

Let us remark here that formulas \eqref{XVIII} and \eqref{candual-matrix} will not always be applicable when we replace the canonical dual $\seq[y]$ with another dual $\seq[z]$, even if $E$ satisfies the MRC for $\seq[x]$ and $\seq[z],$ see  Example~\ref{abc_not_d}(a) below. This is in contrast to the
use of the canonical dual, for which \eqref{XVIII} and \eqref{candual-matrix} can be used (by Theorem~\ref{thm2.5}) as soon as $E$ satisfies the MRC for $\seq[x]$.
\hfill $\Box$
\end{ex}

\smallskip

As shown in the above examples,
replacing the canonical dual
by another dual frame in the considerations in Section \ref{section2} does not  lead to the same conclusions in general. The purpose of this section is to investigate the case when an arbitrary dual frame is used and the differences which occur in comparison with the canonical case.

 Let $\seq[x]$ be a frame for $\h$ and let $\seq[z]$ be a dual frame of $\seq[x]$.
Let $E=\{1,2,\ldots,k\}$.
We denote
\begin{equation}\label{matrix-xz}
A_{X,Z,E}=\left[\begin{array}{cccc}
\langle z_{1},x_{1}\rangle&\langle z_{2},x_{1}\rangle&\ldots&\langle z_{k},x_{1}\rangle\\
\langle z_{1},x_{2}\rangle&\langle z_{2},x_{2}\rangle&\ldots&\langle z_{k},x_{2}\rangle\\
\vdots&\vdots& &\vdots\\
\langle z_{1},x_{k}\rangle&\langle z_{2},x_{k}\rangle&\ldots&\langle z_{k},x_{k}\rangle
\end{array}
\right]-I.
\end{equation}
Our main aim is to consider relations between the following statements:
\begin{enumerate}
  \item[(A)] $E$ satisfies the MRC for $\seq[x]$.
  \item[(A')] $E$ satisfies the MRC for $\seq[x]$ and $\seq[z]$.
  \item[(B)] The matrix $A_{X,Z,E}$ is invertible.
  \item[(C)] The operator $I-\sum_{i=1}^k\theta_{z_i,x_i}$ is invertible.
  \item[(D)] Using initialization $v_n^0=z_n,n\in \Bbb N$, the iterations in \eqref{candual-iterations-alpha} and \eqref{candual-iterations} for $j$ from $1$ to $k$ are well defined.
\end{enumerate}

For the case when $\seq[z]$ is the canonical dual
 of $\seq[x]$, all the above five statements are mutually equivalent ((A) and (A') are clearly equivalent, and for the rest see \cite[Proposition~2.4]{AB}, Proposition~\ref{iterations_canonical}, and Remark~\ref{d_implies_a}). Here we show that in the case of an arbitrary dual frame $\seq[z]$ of $\seq[x],$ some of these implications still hold, but not all of them.
We will prove the following theorem.

\begin{thm}\label{mainthm}
Let $\seq[x]$ be a frame for $\h$ and let $\seq[z]$ be a dual frame of $\seq[x]$. Let $E=\{1,2,\ldots,k\}$. Then
\begin{equation}\label{dcba}
\mbox{(D)}\underset{\nLeftarrow}{\Rightarrow}\mbox{(C)}\Leftrightarrow\mbox{(B)}\underset{\nLeftarrow}\Rightarrow\mbox{(A')}\underset{\nLeftarrow}\Rightarrow\mbox{(A)}.
\end{equation}
\end{thm}

The proof of the above theorem is postponed for the end of the section as it will be a collection of several statements and examples.
We will need the following well-known result.

\begin{lemma}\label{leminvert}
For operators $T:\h_1\to\h_2$ and $S:\h_2\to\h_1$, invertibility of $I_{\h_2}-TS$ on $\h_2$ is equivalent to invertibility of $I_{\h_1}-ST$ on $\h_1$, and in the case of invertibility one has that
 $(I_{\h_2}-TS)^{-1}=I_{\h_2} + T(I_{\h_1} -ST)^{-1} S$.
\end{lemma}

\vspace{.1in}
The proof of the equivalence of (B) and (C) can be done basically in the same way as the proof of \cite[Proposition~2.4 (c)$\Leftrightarrow$(d)]{AB}, but for convenience of the readers we include a short direct proof.

\begin{lemma}\label{invertibility}
Let $E$ be a finite set of indices. Consider arbitrary sets $\seqE[x]$ and $\seqE[z]$ with elements from $\h$.
Then the operator $I-\sum_{i\in E} \theta_{z_i,x_i}$ is invertible on $\h$ if and only if
the corresponding matrix $A_{X,Z,E}$ is invertible.
\end{lemma}
\begin{proof}
Without loss of generality, consider $E=\{1,2,\ldots,k\}$.
Let $U,V:\h\to \Bbb C^k$ denote the analysis operators of the Bessel sequences $\seqE[x]$ and $\seqE[z]$ in $\h$, respectively.
Then
$I-\sum_{i\in E} \theta_{z_i,x_i} =I-V^*U\in\Bbb B(\h)$, while $A_{X,Z,E}$ is the matrix representation of $UV^*-I$ in the canonical basis for $\Bbb C^k.$
To complete the proof, apply Lemma \ref{leminvert} to conclude that $I-V^*U$ is invertible if and only if $UV^*-I$ is invertible.
\end{proof}

We now prove that the condition (C) enables us to define a dual frame for $(x_n)_{n\in E^c}$ as in \eqref{candual-operator} and, since (B)$\Leftrightarrow$(C) by the previous lemma, we can also define a dual as in \eqref{candual-matrix}. These two constructions will give the same dual frame.
In particular, this will imply that (C)$\Rightarrow$(A').
Note that another proof of
(B)$\Rightarrow$(A) can be found in \cite[Proposition~2.8]{AB}.

\begin{prop}\label{prop1} Let $\seq[x]$ be a frame for $\h,$ $\seq[z]$ be its dual frame, and $E=\{1,2,\ldots,k\}$.
Assume that $I-\sum_{i=1}^k\theta_{z_i,x_i}$ is invertible on $\h$. Then
 $\seqEc[x]$ is a frame for $\h$ and
the sequence
\begin{equation}\label{altdual-operator}
\omega_n:=(I-\sum_{i=1}^k\theta_{z_i,x_i})^{-1}z_n,\quad n\in E^c,
\end{equation}
is a dual frame of $\seqEc[x]$.
In particular, $E$ satisfies the MRC for $\seq[x]$ and for $\seq[z]$.
Further, if we let
\begin{equation}\label{alpha_alt_dual}
\left[
\begin{array}{cccc}
\alpha_{n1}&
\alpha_{n2}&
\ldots&
\alpha_{nk}
\end{array}
\right]^T=
A_{X,Z,E}^{-1}\left[
\begin{array}{cccc}
\langle z_n,x_{1}\rangle&
\langle z_n,x_{2}\rangle&
\ldots&
\langle z_n,x_{k}\rangle
\end{array}
\right]^T
\end{equation}
for $n\in E^c,$  then
\begin{equation}\label{altdual-matrix}
\omega_n=z_n-\sum_{i=1}^k\alpha_{ni}z_{i},\quad n\in E^c.
\end{equation}
\end{prop}

\begin{proof}
Observe first that $\seqEc[x]$ and $\seqEc[\omega]$ are Bessel sequences.
 Furthermore, for every $h\in \h$ we have
$$h=\sum_{n=1}^\infty \<h,x_n\>z_n
= \sum_{i\in E} \<h,x_i\>z_i+\sum_{n\in E^c} \<h,x_n\>z_n=
\sum_{i\in E}\theta_{z_i,x_i} (h)  +\sum_{n\in E^c} \<h,x_n\>z_n,$$
and thus
$$(I-\sum_{i\in E}\theta_{z_i, x_i})h= \sum_{n\in E^c} \<h,x_n\> z_n.$$
Then, by \eqref{altdual-operator}, we can write
$$h=\sum_{n\in E^c} \<h,x_n\> \omega_n,\quad h\in \h.$$
This means that  the Bessel sequences $\seqEc[x]$ and $\seqEc[\omega]$ are dual to each other and thus they are frames for $\h$, which implies that $\seqEc[z]$ is also a frame for $\h.$
In particular, $E$ satisfies the MRC for both $\seq[x]$ and $\seq[z]$.

Let us now prove that $\seqEc[\omega]$ can be presented in the form \eqref{altdual-matrix}. First, by Lemma~\ref{invertibility}, $A_{X,Z,E}$ is invertible so the coefficients $\alpha_{ni}$ from \eqref{alpha_alt_dual} are well defined.
Let $U,V:\h\to \Bbb C^k$ denote the analysis operators of the Bessel sequences
$\seqE[x]$ and $\seqE[z]$, respectively. Then,
using Lemma \ref{leminvert},
$$\omega_n=(I-\sum_{i=1}^k\theta_{z_i,x_i})^{-1}z_n
=(I-V^*U)^{-1}z_n=(I-V^*(UV^* - I)^{-1}U)z_n.
$$
Observe that the matrix $A_{X,Z,E}^{-1}$ is the matrix representation of the operator $(UV^*-I)^{-1}$ in the canonical basis for $\Bbb C^k,$ while the vector representation of $Uz_n$ in the canonical basis for $\Bbb C^k$ is  $\left[
\begin{array}{cccc}
\langle z_n,x_{1}\rangle & \langle z_n,x_{2}\rangle& \ldots & \langle z_n,x_{k}\rangle
\end{array}
\right]^T.$ Therefore, by \eqref{alpha_alt_dual},
$$(UV^*-I)^{-1}Uz_n =\left[
\begin{array}{cccc}
\alpha_{n1}&
\alpha_{n2}&
\ldots&
\alpha_{nk}
\end{array}
\right]^T,$$
so for all $n\in E^c$ we get
\begin{eqnarray*}
\omega_n&=&z_n-V^*\left[
\begin{array}{cccc}
\alpha_{n1}&
\alpha_{n2}&
\ldots&
\alpha_{nk}
\end{array}
\right]^T =z_n-\sum_{i=1}^k\alpha_{ni}z_{i},
\end{eqnarray*}
which is \eqref{altdual-matrix}.
\end{proof}

The next example shows that reverses of implications in \eqref{dcba} do not hold.

\begin{ex}\label{abc_not_d}
Let $\seq[e]$ be an orthonormal basis of $\h.$ Consider the frame
$\seq[x]=(e_1, e_1, e_1, e_2, e_3, e_4, \ldots)$ and its (non-canonical) dual frame
$$\seq[z]=(e_1, -\frac{1}{2} e_1, \frac{1}{2}e_1, e_2, e_3, e_4,\ldots).$$

(a) The set $E=\{1\}$ satisfies (A') but not (C) since $I-\theta_{z_1,x_1}=I-\theta_{e_1,e_1}$ is not invertible.

(b) The set $E=\{1,2\}$ satisfies (C), because the operator
$I-\sum_{i\in E}\theta_{z_i,x_i}=I-\frac{1}{2}\theta_{e_1,e_1}$
is invertible on $\h$.  Thus, we can apply Proposition \ref{prop1} to obtain a dual frame of $(x_n)_{n\in E^c}$.
However, if we take $v_n^0=z_n,$ then the first step in \eqref{candual-iterations-alpha} cannot be done since  $\la z_1,x_1\ra=1$, so $\alpha_n^1=\frac{\la z_n,x_1\ra}{1-\la z_1,x_1\ra}$ for $n\geq 2$ is not well defined and thus (D) does not hold.
\end{ex}

\smallskip

While the use of the canonical dual
 guarantees  that the iterative scheme in (\ref{candual-iterations-alpha})-(\ref{candual-iterations}) is well defined (under the assumption for the MRC of $E$),
 Example \ref{abc_not_d} shows that, if we want to do iterations as in (\ref{candual-iterations-alpha})-\eqref{candual-iterations} but starting with an arbitrary dual frame instead of the canonical dual of $(x_n)_{n=1}^\infty,$ we have to be more careful and in each step check whether
 $\<v_j^{j-1},x_j\>$ is 1 - if it is so, then we cannot make this step and the iterative process stops.
In the case when all the iteration steps work, it is natural to pose the question how the obtained iterative sequence relates to $(\omega_n)_{n\in E^c}$ from Proposition~\ref{prop1}; the next proposition clarifies that these two sequences would be the same.

\begin{prop}\label{iterations-alt-dual}
Let $\seq[x]$ be a frame for $\h,$ $\seq[z]$ be its dual frame, and $E=\{1,2,\ldots,k\}$.
Let $$u_n^0:=z_n, \quad n\in\Bbb N.$$ For $j\in E,$ supposing that $(u_n^{j-1})_{n=j}^\infty$ is well defined and that
$\la u_j^{j-1},x_j\ra\neq 1,$
we define
\begin{equation}\label{iter_prop35}
u_n^j:=u_n^{j-1} + \frac{\la u_n^{j-1},x_j\ra}{1-\la u_j^{j-1},x_j\ra} u_j^{j-1}, \quad n\geq j+1.
\end{equation}
Then, for all $j\in E$ for which the sequence $\seqgr[u^j]{j+1}{\infty}$ is well defined, the operator $I-\sum_{i=1}^j \theta_{z_i,x_i}$ is invertible on $\h$ and
\begin{equation}\label{iter2_prop35}
u_n^j=(I-\sum_{i=1}^j \theta_{z_i,x_i})^{-1}z_n,\quad n\ge j+1.
\end{equation}
Furthermore, for  these values of $j$ we also have that  $\seqgr[x]{j+1}{\infty}$ is a frame for $\h$, $\seqgr[u^j]{j+1}{\infty}$ is a dual frame of $\seqgr[x]{j+1}{\infty}$, and the set $\{1,2,\ldots,j\}$ satisfies the MRC for $\seq[x]$ and $\seq[z].$
\end{prop}

\begin{proof}
First note that when the sequence $(u_n^j)_{n=j+1}^\infty$ in \eqref{iter_prop35} is well defined for some $j\in E$, it assumes that $(u_n^k)_{n=k+1}^\infty$ are well defined for all $k=1,\ldots,j-1.$
In that case, by  \eqref{iter_prop35} and the known statement with respect to \eqref{xy-inverse}, for each $k\in\{1,\ldots,j\}$, the operator $I-\theta_{u_k^{k-1},x_k}$ is invertible and
\begin{equation}\label{step}
u_n^k=(I-\theta_{u_k^{k-1},x_k})^{-1}u_n^{k-1}, \quad n\geq k+1.
\end{equation}

Let us now show the invertibility of  $I-\sum_{i=1}^j \theta_{z_i,x_i}$ and the validity of \eqref{iter2_prop35}, using induction on the values of $j$ in $E$.

If the sequence $(u_n^1)_{n\ge 2}$ in \eqref{iter_prop35} is well defined, then by the observation above we have that  $I- \theta_{z_1,x_1}$ is invertible and
 $u_n^1=(I-\theta_{z_1,x_1})^{-1}z_n$ for all $n\ge 2$, so \eqref{iter2_prop35} holds.

To proceed with induction, assume that the statement is proven for the case when
$(u_n^{j-1})_{n=j}^\infty$ is well defined for some $j\in E$, $j>1$.
Now, suppose that $(u_n^{j})_{n=j+1}^\infty$ is well defined (and thus that  $(u_n^{j-1})_{n=j}^\infty$ is also well defined).
Then using  the induction step $j-1$ and the observation at the beginning of the proof, one can write
\begin{eqnarray*}
I-\sum_{i=1}^{j}\theta_{z_i,x_i} &=&
I-\sum_{i=1}^{j-1}\theta_{z_i,x_i} -\theta_{(I-\sum_{i=1}^{j-1}\theta_{z_i,x_i})u_{j}^{j-1},x_{j}}\\
&=&(I-\sum_{i=1}^{j-1}\theta_{z_i,x_i} )(I-\theta_{u_j^{j-1},x_j})
\end{eqnarray*}
and  conclude that $I-\sum_{i=1}^{j}\theta_{z_i,x_i} $ is invertible on $\h$. Furthermore, for every $n\ge j+1$, we have
\begin{eqnarray*}
u_n^j&\stackrel{\eqref{step}}{=}&(I-\theta_{u_j^{j-1},x_j})^{-1}u_n^{j-1}
=(I-\theta_{u_j^{j-1},x_j})^{-1} (I-\sum_{i=1}^{j-1}\theta_{z_i,x_i})^{-1}z_n\\
&=&(I-\sum_{i=1}^{j}\theta_{z_i,x_i})^{-1}z_n.
\end{eqnarray*}

Now the remaining statements follow from Proposition~\ref{prop1}.
\end{proof}

We now have a complete proof of Theorem~\ref{mainthm}.
\vspace{.1in}

\noindent \emph{Proof of Theorem~\ref{mainthm}}.
By Lemma~\ref{invertibility}, (B)$\Leftrightarrow$(C).
It follows from Proposition~\ref{prop1} that $\mbox{(C)}\Rightarrow\mbox{(A')}$, while Proposition~\ref{iterations-alt-dual} gives  (D)$\Rightarrow$(C).
For (A)$\nRightarrow$(A') see Example~\ref{ex1}, and for (A')$\nRightarrow$(C)$\nRightarrow$(D) see Example~\ref{abc_not_d}.
\qed

\vspace{.1in}
 The next two statements provide classes of dual frames $(z_n)_{n=1}^\infty$ of $(x_n)_{n=1}^\infty$, for which (B), resp. (D), holds.
 We will use the known result (see, e.g., \cite[Theorem 6.3.7]{Cbook}) that all the dual frames $\seq[z]$ of a given frame $\seq[x]$ for $\h$ can be written as
 \begin{equation}\label{zdualfr}
 z_n = y_n + q_n  -\sum_{i=1}^\infty \<y_n , x_i\> q_i, \ n\in \mn,
 \end{equation}
where $\seq[y]$ denotes the canonical dual of $\seq[x]$ and $\seq[q]$ is a Bessel sequence in $\h$.

\begin{prop}\label{prop11}
Let $(x_n)_{n=1}^\infty$ be a frame for $\h$ and let $(y_n)_{n=1}^\infty$ be its canonical dual.
Let $E=\{1,2,\ldots,k\}$ satisfy the MRC for $\seq[x]$ and let $Q=(q_n)_{n=1}^\infty$ be a sequence in $\h$ such that $q_n=0$ for $n\in E^c$. Then the dual frame $\seq[z]$ of $\seq[x]$ determined by (\ref{zdualfr}) satisfies (B) 
if and only if the matrix $A_{X,Q,E}$ is invertible.
\end{prop}

\bp
First note that $Q$ is a Bessel sequence in $\h$ and thus  $\seq[z]$ is a dual frame of $\seq[x]$. 
By Lemma \ref{invertibility}, $\seq[z]$ satisfies (B)  if and only if 
the operator $I-\sum_{i=1}^k\theta_{z_i,x_i}$ is invertible. 
Now we have
 $$z_n = y_n + q_n  -\sum_{i=1}^k \<y_n , x_i\> q_i=q_n+(I-\sum_{i=1}^k\theta_{q_i,x_i})y_n, \quad n\in \mn.$$
Denote $T:= I-\sum_{i=1}^k\theta_{q_i,x_i}$.
Then
\begin{eqnarray*}
I-\sum_{i=1}^k\theta_{z_i,x_i} &=&
I-\sum_{i=1}^k \theta_{q_i+Ty_i,x_i}
=
I-\sum_{i=1}^k \theta_{q_i,x_i}-\sum_{i=1}^k \theta_{Ty_i,x_i}\\
&=&
T-T\sum_{i=1}^k \theta_{y_i,x_i}
=
T(I -\sum_{i=1}^k \theta_{y_i,x_i}).
\end{eqnarray*}
By \cite[Proposition~2.4]{AB}, the operator $I -\sum_{i=1}^k \theta_{y_i,x_i}$ is invertible.
Therefore,  $I-\sum_{i=1}^k\theta_{z_i,x_i}$ is invertible if and only if $T$ is invertible. 
Finally, by Lemma~\ref{invertibility},  
$ T=I -\sum_{i=1}^k \theta_{q_i,x_i} $ is invertible if and only if the matrix $A_{X,Q,E}$ is invertible. 
\ep

\vspace{.1in}
As a particular simple example of a sequence $Q$, satisfying the conditions of the above proposition, consider for example any sequence $\seq[q]$ with $q_{i_0}\neq 0$ for some ${i_0}\in E$ such that $\<q_{i_0},x_{i_0}\>\neq 1$,
and $q_j=0$ for $j\in \mn\setminus \{i_0\}$.

Note that the matrix $A_{X,Z,E}$ can be invertible
even if a dual frame $\seq[z]$ of $\seq[x]$ is associated to a sequence $\seq[q]$ which does not satisfy the condition  $q_n=0$ for $n\in E^c$.
For example, let $\seq[q]$ be the canonical dual frame $\seq[y]$ of $\seq[x]$. Then for every $E$ with the MRC for $\seq[x]$ it holds that $(q_n)_{n\in E^c}$ is not a zero sequence,  but the matrix $A_{X,Q,E}=A_{X,Y,E}$ is invertible by \cite[Proposition~2.4]{AB}. 
Observe that in this case we have $\seq[z]=\seq[y]$.

\vspace{.1in}
Below we determine a class of dual frames for which the iterative procedure in Proposition \ref{iterations-alt-dual} works for all steps:

\begin{prop}\label{prop12}
Let $\seq[x]$ be a frame for $\h$,  $\seq[y]$ be  its  canonical dual, and let $E=\{1,2,\ldots,k\}$ satisfy the MRC for $\seq[x]$.
Let $(q_n)_{n=1}^\infty$ be a Bessel sequence in $\h$ such that $q_n=0$ for $n\in E^c$ and
$q_n \perp {\rm span} (x_i)_{i\in E}$ for  $n\in E$.
If  $\seq[z]$ is the dual frame of $\seq[x]$ defined  by \eqref{zdualfr}, then the iterative procedure in Proposition~\ref{iterations-alt-dual} works for all $j\in E$.
\end{prop}
\bp
To prove our statement, it is enough to show that
$\<u_j^{j-1}, x_j\>\neq 1$ holds for every $j\in E$, where $u_j^{j-1}$ is as in Proposition~\ref{iterations-alt-dual}.
To do this, we involve the sequences $(v_n^j)_{n=j+1}^\infty$ for $j\in \{0,1,2,\ldots,k\}$ determined by Proposition  \ref{iterations_canonical}.

Let us first show that for every $j\in \{0,1,2,\ldots,k\}$ we have
\begin{equation}\label{uvrelation}
\<u_n^j, x_p\> = \<v_n^j, x_p\>, \quad
n,p\in \{j+1, j+2,\ldots, k\}.
\end{equation}
We prove this by induction on $j$. Let $j=0$. Then
 for every $n,p\in \{1,2,\ldots, k\}$ we have
$$\<u_n^0,x_p\> = \<z_n, x_p\>
 =\<y_n + q_n  -\sum_{i\in E} \<y_n , x_i\> q_i, x_p\>=
 \<y_n, x_p\>=\<v_n^0,x_p\>.$$
Now assume that (\ref{uvrelation}) holds for some  $j<k$, and consider the step $j+1$.
Using (\ref{iter_prop35}), (\ref{candual-iterations-alpha})-(\ref{candual-iterations}), and the induction assumption for step $j$,
  for $n,p\in \{j+2, j+3,\ldots, k\}$  we get
 \begin{eqnarray*}
\<u_n^{j+1},x_p\> &=& \<u_n^j + \frac{\la u_n^j,x_{j+1}\ra}{1-\la u_{j+1}^j,x_{j+1}\ra} u_{j+1}^j, x_p\>\\
& =&
\<v_n^j + \frac{\la v_n^j,x_{j+1}\ra}{1-\la v_{j+1}^j,x_{j+1}\ra} v_{j+1}^j, x_p\>= \<v_n^{j+1},x_p\>.
 \end{eqnarray*}
This proves (\ref{uvrelation}).
 In particular, (\ref{uvrelation}) implies that
 $$\<u_j^{j-1}, x_j\>=\<v_j^{j-1}, x_j\>,\quad j\in E.$$
By Proposition  \ref{iterations_canonical}, $\<v_j^{j-1}, x_j\>\neq 1$ for every $j\in E$, which completes the proof.
\ep

%%%%%%%%%%%%%%%%%%%%%%%
\section{Implementation and computational efficiency}

In this section we examine the computational efficiency of the approaches in Sections~\ref{section2}  and \ref{section3}  in order to justify the relevance of the considered procedures from computational view point.
In the finite-dimensional case, we have implemented the algorithms of Propositions~\ref{prop1} and  \ref{iterations-alt-dual} which provide constructions of a dual frame of the reduced frame $(x_n)_{n\in E^C}$ (resp. Proposition \ref{iterations_canonical} and Theorem \ref{thm2.5}(ii) which provide constructions of the canonical dual frame).
The scripts are available on \sloppy \url{http://dtstoeva.podserver.info/ReconstructionUnderFrameErasures.html} and \url{https://www.nt.tuwien.ac.at/downloads/}. 
The programming is done under the MATLAB environment, using also frame-commands from LTAFT\footnote{The Large Time-Frequency Analysis Toolbox (a Matlab/Octave open source toolbox for dealing with time-frequency analysis and synthesis),
\url{http://ltfat.org/}, see e.g. \cite{ltfatnote030,ltfatnote015}.}.

We tested the efficiency of the scripts for various frames $X=(x_n)_{n=1}^N$ varying the number $N$ of the frame elements, the dimension $r$ of the space, the redundancy $N/r$,
 the cardinality $k$ of the erasure set $E=\{1,2,\ldots,k\}$, as well as the starting dual frame $Z$ of $X$.
The elapsed time recorded in Table 1 is in seconds, measured
 using the MATLAB tic-toc functions.

We compare the time for computing the canonical dual of the reduced frame $(x_n)_{n\in E^c}$ via the code
based on Proposition~\ref{iterations_canonical} ($t_1$ in Table 1),
via the code based on  Theorem \ref{thm2.5}(ii) ($t_2$  in Table 1), and via the pseudo-inverse\footnote{The pseudo-inverse approach is based on the fact that the synthesis operator of the canonical dual of a frame $X$ is the adjoint of the Moore-Penrose pseudoinverse of the synthesis operator of $X$ \cite[Theorem 1.6.6]{Cbook}.}
approach  ($t_3$ in Table 1)  using the MATLAB \emph{pinv}-function\footnote{Our first aim was to do comparison with the LTFAT function \emph{framedual} for computing the canonical dual, but since for general frames \emph{framedual} uses the pseudo-inverse approach calling the \emph{pinv}-function from MATLAB, we compare directly to the \emph{pinv}-function to avoid unnecessary delay.}.
For the same frame $X$, for which the aforementioned tests were performed concerning the canonical dual,
we also compare the time  for computing another dual frame of the reduced frame $(x_n)_{n\in E^c}$ via formula (\ref{altdual-matrix}) in Propositions \ref{prop1}
($t_4$ in Table 1) and via the iterative approach in Proposition \ref{iterations-alt-dual} ($t_5$ in Table 1) starting with a non-canonical dual $Z$ of $X$.
In Table~1 we present some results from the tests - the execution time of the considered procedures and the respective errors. For each procedure in a test, the respective
error $e_i$
is computed using the MATLAB 2-norm function of
the matrix $V^* U-I$, where $U$ denotes the analysis operator of  $(x_n)_{n\in E^c}$ and $V$ denotes  the analysis operator of the constructed dual  frame $(v_n)_{n\in E^c}$ via the considered procedure.
Through the results of the tests, on the one hand, one can compare the performance of the two procedures
 - via inversion of the matrix $A_{X,Z,E}$ and via the iterative algorithm,
and on the other hand, one can compare the performance of the canonical dual versus  another dual frame for the desired constructions.

Concerning Tests 1-8: With $N$ and $r$ chosen by the user, the test-program produces a frame $X=(x_n)_{n=1}^N$ with random elements; then the user can make choices for $k$ until the program verifies that $E$ has the MRC for $X$, and for this set $E$ the program measures $t_1$-$t_3$; 
then a dual frame $Z_1$ of $X$ is randomly chosen and $t_4(Z_1)$ and $t_5(Z_1)$ 
are measured;
and also another dual frame $Z_2$ of $X$ is randomly chosen and the respective $t_4(Z_2)$ and $t_5(Z_2)$ are measured. 

Concerning Test 9:
We took an explicit simple frame $X$, its canonical dual $Y$, and an explicit simple non-canonical dual frame $Z_1$ of $X$ (aiming to include a case where Propositions \ref{iterations-alt-dual} and \ref{prop1} do not apply), while the user can make choices for $k$ until the program verifies that $E$ has the MRC for $X$.  For this set $E$ the program measures $t_1$-$t_3$ and $t_4(Z_1)$-$t_5(Z_1)$ (we have chosen a value of $k$ such that the iterative procedure of Proposition \ref{iterations-alt-dual} cannot be completed and where the MRC does not hold for $Z_1$ so Proposition \ref{prop1} does not apply), and finally another dual frame $Z_2$ of $X$ is randomly chosen and the respective times $t_4(Z_2)$-$t_5(Z_2)$ are measured.

\begin{table}[h!] \label{tab:table1new}
  \begin{center}
        {\tiny
          \begin{tabular}{|l|r|r|r|r|r|r|r|r|r|}
         \hline
       \ & Test 1  & Test 2 & Test 3 & Test 4 & Test 5 & Test 6 & Test 7 & Test 8 & Test 9\\
          \hline
     \hspace{-.09in}  $N$\hspace{-.09in}  & 6000 & 6000 & 6000 & 7000 & 5000 & 8000 & 8000 & 8000 & 
      3010 \\
               \hline
     \hspace{-.09in}  $r$ \hspace{-.09in} & 4000 &4000& 4000 & 4000 & 4000 & 200  & 2000 & 6000 & 3000 \\
               \hline
     \hspace{-.09in}   $k$ \hspace{-.09in} & 200& 300 & 500 & 50 & 200 & 80 & 200& 500 & 4 \\
   \Xhline{1pt}

  \hspace{-.09in}   $t_1$ \hspace{-.09in} 
       & \hspace{-.05in}    \textbf{32.1367} \hspace{-.09in} 
       &  \hspace{-.05in}  \textbf{\textcolor{blue}{48.4165}} \hspace{-.09in}
       &  \hspace{-.05in}  79.8354 \hspace{-.09in}
        &  \hspace{-.05in}   10.5032 \hspace{-.09in}
         & \hspace{-.05in}   \textbf{26.3473} \hspace{-.09in}
         &  \hspace{-.05in}  \textbf{6.2474} \hspace{-.09in}
         &  \hspace{-.05in}  15.2432 \hspace{-.09in}
         &  \hspace{-.05in}  \textbf{147.4556} \hspace{-.09in}
         & \hspace{-.07in}     \textbf{0.3145}  \hspace{-.09in}    \\
               \hline

   \hspace{-.09in}  $t_2$  \hspace{-.09in}  
       & \hspace{-.05in}   36.9786 \hspace{-.09in}
       &  \hspace{-.05in}  74.8509 \hspace{-.09in}
       &  \hspace{-.05in}  80.7513 \hspace{-.09in}
        & \hspace{-.05in}    \textbf{ 8.8483} \hspace{-.09in}
        & \hspace{-.05in}   29.7811 \hspace{-.09in}
        & \hspace{-.05in}   6.4447 \hspace{-.09in}
         & \hspace{-.05in}   15.2507 \hspace{-.09in}
         &  \hspace{-.05in}  155.4476 \hspace{-.09in}
         &  \hspace{-.07in}    0.3407 \hspace{-.09in}  \\
               \hline

    \hspace{-.09in}  $t_3$  \hspace{-.09in}  
       &  \hspace{-.05in}  80.1901 \hspace{-.09in}
       &   \hspace{-.05in}       141.7107 \hspace{-.09in}
       & \hspace{-.05in}   \textbf{\textcolor{blue}{77.0104}} \hspace{-.09in}
       &  \hspace{-.05in}   65.2870 \hspace{-.09in}
       &  \hspace{-.05in}  63.9953 \hspace{-.09in}
       &  \hspace{-.05in}     11.5692 \hspace{-.09in}
       & \hspace{-.05in}   \textbf{\textcolor{blue}{11.7546}} \hspace{-.09in}
       &  \hspace{-.05in}  374.6021 \hspace{-.09in}
        &  \hspace{-.07in}  22.6884 \hspace{-.09in}  \\
                \Xhline{0.9pt}

 \hspace{-.09in}  $t_4(Z_1)$  \hspace{-.09in}
       & \hspace{-.05in}    \textbf{\textcolor{blue}{31.9943}} \hspace{-.09in}
       &    \hspace{-.05in}      \textbf{69.8041} \hspace{-.09in}
       &  \hspace{-.05in}  \textbf{79.6817} \hspace{-.09in}
       &  \hspace{-.05in}      9.5098 \hspace{-.09in}
        &  \hspace{-.05in}     26.3567 \hspace{-.09in}
        &  \hspace{-.05in}   6.1913 \hspace{-.09in}
        & \hspace{-.05in}   \textbf{15.2210} \hspace{-.09in}
        &  \hspace{-.05in}   149.8696 \hspace{-.09in} 
           &  \hspace{-.07in}  stop at j=3 \hspace{-.09in}  \\
               \hline

  \hspace{-.09in}  $t_5(Z_1)$  \hspace{-.09in}
       &  \hspace{-.05in}  32.0460 \hspace{-.09in}
       &  \hspace{-.05in}  77.7833 \hspace{-.09in}
       & \hspace{-.05in}   81.3535 \hspace{-.09in}
       & \hspace{-.05in}    \textbf{9.2100} \hspace{-.09in}
       &  \hspace{-.05in}  \textbf{26.3251} \hspace{-.09in}
       & \hspace{-.05in}   \textbf{\textcolor{blue}{6.0000}} \hspace{-.09in}
        & \hspace{-.05in}     15.3378 \hspace{-.09in}
        &  \hspace{-.05in}  \textbf{\textcolor{blue}{144.2378}} \hspace{-.09in} 
    &   \hspace{-.07in}  no\,Prop.\,3.5 \hspace{-.09in} \\
                           \Xhline{0.9pt}

 \hspace{-.09in}  $t_4(Z_2)$  \hspace{-.09in}
       & \hspace{-.05in}    36.4463 \hspace{-.09in}
       &  \hspace{-.05in}  70.1102 \hspace{-.09in}
       & \hspace{-.05in}   80.9745 \hspace{-.09in}
       &  \hspace{-.05in}   11.0663 \hspace{-.09in}
	    & \hspace{-.05in}    \textbf{\textcolor{blue}{26.2390}} \hspace{-.09in}
       & \hspace{-.05in}   7.0013 \hspace{-.09in}
        & \hspace{-.05in}   \textbf{15.3719} \hspace{-.09in}
        & \hspace{-.05in}   151.9858 \hspace{-.09in}
        & \hspace{-.07in}    0.2891 \hspace{-.09in}  \\
               \hline

 \hspace{-.09in}  $t_5(Z_2)$  \hspace{-.09in}
       & \hspace{-.05in}    \textbf{32.1301}  \hspace{-.09in}
       & \hspace{-.05in}   \textbf{57.7437} \hspace{-.09in}
       & \hspace{-.05in}   \textbf{ 79.2500} \hspace{-.09in}
       &  \hspace{-.05in}   \textbf{\textcolor{blue}{8.6892}} \hspace{-.09in}
        &  \hspace{-.05in}  30.0706 \hspace{-.09in}
        &  \hspace{-.05in}  \textbf{6.0525} \hspace{-.09in}
        &  \hspace{-.05in}  17.9037 \hspace{-.09in}
        &  \hspace{-.05in}  \textbf{150.2319} \hspace{-.09in}
        &   \hspace{-.07in}  \textbf{\textcolor{blue}{0.2522}} \hspace{-.09in}  \\
       \Xhline{0.9pt}

   \hspace{-.09in}   $e_1$  \hspace{-.09in}
       & \hspace{-.05in}   6.3344e-14 \hspace{-.09in}
       & \hspace{-.05in}   7.7079e-14 \hspace{-.09in}
       &  \hspace{-.05in}   1.1377e-13 \hspace{-.09in}
       &  \hspace{-.05in}  3.0507e-14 \hspace{-.09in}
       &  \hspace{-.05in}  2.1044e-13 \hspace{-.09in}
       & \hspace{-.05in}   1.5939e-14 \hspace{-.09in}
       &  \hspace{-.05in}  2.0370e-14 \hspace{-.09in}
       &  \hspace{-.05in}  2.1916e-13 \hspace{-.09in}
       &  \hspace{-.07in}   2.2204e-16 \hspace{-.09in} \\
               \hline

  \hspace{-.09in}   $e_2$  \hspace{-.09in}
       & \hspace{-.05in}   6.2526e-14 \hspace{-.09in}
       &  \hspace{-.05in}  7.6374e-14 \hspace{-.09in}
       &  \hspace{-.05in}   1.1156e-13 \hspace{-.09in}
       &  \hspace{-.05in}  3.0589e-14 \hspace{-.09in}
       &  \hspace{-.05in}    2.0469e-13 \hspace{-.09in}
        &  \hspace{-.05in}    1.5934e-14 \hspace{-.09in}
        &  \hspace{-.05in}  2.0358e-14 \hspace{-.09in}
        &  \hspace{-.05in}     2.1441e-13 \hspace{-.09in}
         &  \hspace{-.07in}    2.2204e-16 \hspace{-.09in}  \\
               \hline

   \hspace{-.09in}  $e_3$  \hspace{-.09in}
       &  \hspace{-.05in}    2.7159e-14 \hspace{-.09in}
       & \hspace{-.05in}   3.3145e-14 \hspace{-.09in}
        &   \hspace{-.05in}   3.2635e-14 \hspace{-.09in}
        &   \hspace{-.05in}    2.9561e-14 \hspace{-.09in}
        &  \hspace{-.05in}   4.1709e-14 \hspace{-.09in}
        & \hspace{-.05in}    1.4724e-14 \hspace{-.09in}
        &  \hspace{-.05in}  1.4372e-14 \hspace{-.09in}
        & \hspace{-.05in}   4.5437e-14 \hspace{-.09in}
        & \hspace{-.07in}   2.2204e-16 \hspace{-.09in} \\
               \hline

 \hspace{-.09in}  $e_4(Z_1)$  \hspace{-.09in}
       &  \hspace{-.05in}   2.8838e-04 \hspace{-.09in}
       &  \hspace{-.05in}  0.0010 \hspace{-.09in}
       &  \hspace{-.05in}    8.2355e-04 \hspace{-.09in}
       &  \hspace{-.05in}   1.5013e-05 \hspace{-.09in}
       & \hspace{-.05in}   2.5971e-04 \hspace{-.09in}
       & \hspace{-.05in}   4.5256e-06 \hspace{-.09in}
       &    \hspace{-.05in}  1.4281e-04 \hspace{-.09in}
        &   \hspace{-.05in}  2.3799e-04 \hspace{-.09in}
        &  \hspace{-.07in}   no \hspace{-.05in}  \\
               \hline

  \hspace{-.09in}  $e_5(Z_1)$  \hspace{-.09in}
       &   \hspace{-.05in}  6.8226e-06 \hspace{-.09in}
       &  \hspace{-.05in}  2.4236e-05 \hspace{-.09in}
       &  \hspace{-.05in}  4.8157e-06 \hspace{-.09in}
       &   \hspace{-.05in}   1.3170e-06 \hspace{-.09in}
       &  \hspace{-.05in}   8.7436e-07 \hspace{-.09in}
       &  \hspace{-.05in}  5.6164e-07 \hspace{-.09in}
        & \hspace{-.05in}   1.7502e-06 \hspace{-.09in}
        & \hspace{-.05in}   9.6339e-06 \hspace{-.09in}
        & \hspace{-.07in}   no \hspace{-.05in}  \\
               \hline

 \hspace{-.09in}  $e_4(Z_2)$  \hspace{-.09in}
       & \hspace{-.05in}   2.2986e-05 \hspace{-.09in}
       & \hspace{-.05in}  1.4722e-04 \hspace{-.09in}
        &  \hspace{-.05in}    0.0016 \hspace{-.09in}
        &  \hspace{-.05in}   2.0118e-06 \hspace{-.09in}
         &  \hspace{-.05in}   3.0671e-04 \hspace{-.09in}
         &   \hspace{-.05in}  4.0523e-05 \hspace{-.09in}
          &  \hspace{-.05in}  3.3912e-04 \hspace{-.09in}
          &   \hspace{-.05in}  0.0306 \hspace{-.09in}
           &  \hspace{-.07in}   4.4474e-13 \hspace{-.09in} \\
               \hline

  \hspace{-.09in}  $e_5(Z_2)$  \hspace{-.09in}
       &  \hspace{-.05in}    1.7899e-06 \hspace{-.09in}
       & \hspace{-.05in}   3.2424e-06 \hspace{-.09in}
       &  \hspace{-.05in}    1.9454e-05 \hspace{-.09in}
       &  \hspace{-.05in}    6.4818e-07 \hspace{-.09in}
       &   \hspace{-.05in}   7.2992e-06 \hspace{-.09in}
       & \hspace{-.05in}    4.8560e-07 \hspace{-.09in}
       &  \hspace{-.05in}  1.0330e-05 \hspace{-.09in}
       & \hspace{-.05in}  3.9119e-06 \hspace{-.09in}
       & \hspace{-.07in}  5.6066e-13 \hspace{-.09in} \\
               \hline

    \end{tabular}
    }
    \caption{Tests for the algorithms based on Propositions~\ref{iterations_canonical},  Theorem~\ref{thm2.5}(ii), and the pseudo-inverse approach through the Matlab pinv-function (used in LTFAT) for construction of the canonical dual of the reduced frame (lines $t_1-t_3$ in the table), as well as tests for the algorithms based on Propositions
    \ref{prop1} and \ref{iterations-alt-dual} for construction of a dual of the reduced frame (lines $t_4(Z_1)-t_5(Z_1)$ and lines $t_4(Z_2)-t_5(Z_2)$ in the table). In each test, the shortest executed time is marked in blue, and for each of the three sub-groups $t_1-t_3$, $t_4(Z_1)-t_5(Z_1)$, and $t_4(Z_2)-t_5(Z_2)$ of the respective test, the shortest time in the sub-group is marked with bold style.\
        The scripts used to produce the tests reflected in Table 1 are available on 
 http://dtstoeva.podserver.info/ReconstructionUnderFrameErasures.html and https://www.nt.tuwien.ac.at/downloads/. 
    }
   \end{center}
\end{table}

The tests reflected in Table 1 show the following:
in certain cases, the iterative algorithm in Proposition \ref{iterations_canonical} outperforms Theorem \ref{thm2.5}(ii) (Tests 1-3,5-9)
and the converse also holds (Test 4); in certain cases, both of these approaches perform  (significantly) faster then
the pseudo-inverse approach using the MATLAB pinv-function  (Tests 1,2,4,5,8,9)  and hence also faster than LTFAT for general frames;
in certain cases, some non-canonical duals provide a faster procedure in comparison to the use of the canonical dual, i.e., in certain cases Proposition \ref{iterations-alt-dual} or Proposition \ref{prop1} performs faster than Proposition \ref{iterations-alt-dual} or Theorem \ref{thm2.5}(ii) (Tests 1,4-6,8-9);
the execution time of the algorithms depends also much on the values of $k$, $N$, $r$, $N/r$.

In conclusion, we may say that when we have a dual frame for $(x_n)_{n=1}^N,$ the canonical one or any other,
the algorithms presented in this paper can be efficient for computing a dual frame for the reduced frame $(x_n)_{n\in E^C}$ and can be used to enrich and improve LTFAT. In certain cases, the iterative algorithm  outperforms the procedure involving a matrix inversion, which justifies its consideration. 
In addition to this fact, let us also note that
 the iterative algorithm provides simultaneously a dual frame for all the erasure sets $E_j=\{1,2,\ldots,j\}$, $j=1,2,\ldots,k$, which gives flexibility for simultaneous use of multiple erasure sets.
In certain cases, the use of a non-canonical dual outperforms the use of the canonical dual,
which also justifies the interest to non-canonical duals, in addition to the motivating arguments in the introduction.
The
size of the frame,  its redundancy, the dimension of the spaces, and the cardinality of the erasure set,
may have significant influence on the execution times of the considered procedures.
Further tests and deeper investigation of appropriate dual frames and
optimal values of the aforementioned parameters for efficient computational purposes of each method will be the task of a future work.

\section{Appendix}

Here we provide short pseudocodes of the scripts that were used to produce the tests reflected in Table 1.

\vspace{.1in}\noindent

\begin{itemize}
  \item[I.] Short pseudocode of the script in Table1Tests1til8.m used for producing Tests 1-8:
\begin{enumerate}
  \item Initialize LTFAT in order to use some LTFAT-functions.
  \item Input of $N$ and $r$ by the user.
  \item Random choice of the synthesis matrix $TX$ (size $r \times N$) of a frame $X=(x_n)_{n=1}^N$.
  \item Input of $k$ by the user.
  \item  Verify whether the set $E = \{1,2,\ldots,k\}$ satisfies the MRC for $X$.
\begin{enumerate}
  \item[] If No, new input of $k$ by the user is required.
  \item[] If Yes, the script continues.
\end{enumerate}
 \item  Determine the synthesis matrix $TY$ of the canonical dual of $X$. 
\item Random choice of the synthesis matrix $TZ1$ of a dual frame of $X$.
\item  Random choice of the synthesis matrix $TZ2$ of another dual frame of $X$.
\item  Measure $t_1, t_2, e_1,$ and $e_2$ through the function $CompareSpeedDual$
that involves the codes based on Proposition~\ref{iterations_canonical} and Theorem \ref{thm2.5}(ii).
\item  Measure $t_3$  using the MATLAB function $pinv$ and the respective error $e_3$.
\item  Measure $t_4(Z_1), t_5(Z_1), e_4(Z_1)$, and $e_5(Z_1)$  through the function $CompareSpeedDual$
that involves the codes  based on Propositions \ref{iterations-alt-dual} and \ref{prop1}.
\item  Measure $t_4(Z_2), t_5(Z_2), e_4(Z_2)$, and $e_5(Z_2)$ as in 11.
\item Visualize the values of $t_1$-$t_3$, $t_4(Z_1)$, $t_5(Z_1)$, $t_4(Z_2)$, $t_5(Z_2)$, and the respective errors $e_1$-$e_3$, $e_4(Z_1)$, $e_5(Z_1)$, $e_4(Z_2)$, and $e_5(Z_2)$. 
\end{enumerate}
\item[II.] Short pseudocode of the script in Table1Test9.m used for producing the test reflected in Table 1 Tests 9:

\begin{enumerate}
  \item  Initialize LTFAT in order to use some LTFAT-functions.
   \item Determine $N$ and $r$ (fixed in the script, but can be easily changed by the user in the code for further tests).
\item  Determine the synthesis matrix $TX$ of a specific frame $X$.
\item[4.-6.] The same as in the pseudocode above.
\item [7.] Determine the synthesis matrix $TZ1$ of a specific non-canonical dual frame of $X$.
\item [8.-13.] The same as in the pseudocode above.
\end{enumerate}
\end{itemize}

%%%%%%%%%%%%%%%%%%%%%%%%%%%%%%%%%%%%%%%%%%%%%%%%%%%%%%%%%%%%%%%%%%%%%%%%%%%%

\section*{Acknowledgement}
The authors are grateful to the reviewers for the valuable comments and suggestions.
The authors are supported by the Scientific and Technological Cooperation project Austria--Croatia ``Frames, Reconstruction, and Applications'' (HR 03/2020).
The first author was also supported by the Croatian Science Foundation under the project IP-2016-06-1046.
The second author is grateful for the hospitality of the Department of Mathematics (Faculty of Science, University of Zagreb) during her visit.
She also acknowledges support from the Austrian Science Fund (FWF) under grants P 32055-N31 and  Y 551-N13, and from the
Vienna Science and Technology Fund (WWTF) through the project VRG12-009.

\end{document}